\documentclass[10pt,reqno]{amsart}
\usepackage{amsmath, amscd}
\title[Twisted Sequences]{Twisted Sequences of Extensions}

\author[K. J. Carlin]{Kevin J. Carlin}

\address{Department of Mathematics and Computer Science\\
Assumption College\\
500 Salisbury St.\\
Worcester MA 01609-1296}

\email{kcarlin@assumption.edu}

\date{}


\newcommand{\la}{\lambda}
\newcommand{\cO}{\mathcal{O}}
\newcommand{\cA}{\mathcal{A}}
\newcommand{\fg}{\mathfrak g}

\newcommand{\onto}{\mathrel{\rightarrow\kern-8pt\rightarrow}}
\newcommand{\C}{\mathbb{C}}

\DeclareMathOperator{\Ker}{Ker}
\DeclareMathOperator{\Coker}{Coker}
\DeclareMathOperator{\Coim}{Coim}
\DeclareMathOperator{\Image}{Im}
\DeclareMathOperator{\Ext}{Ext}
\DeclareMathOperator{\E}{E}
\DeclareMathOperator{\h}{H}

\DeclareMathOperator{\Id}{Id}

\newcommand{\tothe}[1]{\strut^{\!\scriptscriptstyle#1}}

\newtheoremstyle{plainsc}%
{6pt plus 1pt minus 1pt}%
{6pt plus 1pt minus 1pt}%
{\itshape}%
{}%
{\scshape}%
{}%
{.5em}%
{}

\newtheoremstyle{definitionsc}%
{6pt plus 1pt minus 1pt}%
{6pt plus 1pt minus 1pt}%
{}%
{}%
{\scshape}%
{:}%
{.5em}%
{}

\theoremstyle{plainsc}
\newtheorem{thm}{Theorem}[section]
\newtheorem{lem}[thm]{Lemma}
\newtheorem{prop}[thm]{Proposition}
\newtheorem{cor}[thm]{Corollary}
\newtheorem{conj}[thm]{Conjecture}

\theoremstyle{definitionsc}

\calclayout

\numberwithin{equation}{thm}

\begin{document}

\begin{abstract}
Gabber and Joseph \cite[\S5]{GJ} introduced a ladder diagram between two natural sequences of extensions. Their diagram is used to produce a \lq twisted' sequence that is applied to old and new results on extension groups in category $\cO$.
\end{abstract}

\maketitle

\section{The Gabber-Joseph Isomorphism}
\label{sec:GJ}

Let $\cA$ be an abelian category with enough projectives. Let $\E\tothe p=\Ext\tothe p_\cA$ (with the convention that $\E\tothe p{=}\,0$ if $p<0$). Let $\h=\E\tothe0=\hom_\cA$. If $\E{}$ is used to represent some $\E\tothe p$, then use the relative notations, $\E\tothe+$ and $\E\tothe-$, to represent $\E\tothe{p+1}$ and $\E\tothe{p-1}$ respectively. 

Suppose that $R$ and $T$ are exact, mutually adjoint endofunctors defined on $\cA$. Let $\theta=R\,T$. The unit of the adjunction $(T,R)$ is $\eta:\Id\rightarrow \theta$ and the co-unit of the adjunction $(R,T)$ is $\epsilon:\theta\rightarrow\Id$. Use these to define the functors,
\begin{align*}
C&=\Coker\eta&D=&\Coim\eta\\
K&=\Ker\epsilon&I=&\Image\epsilon.
\end{align*}
There are also natural transformations, $\iota:I\rightarrow\Id$ and $\pi:\Id\rightarrow D$.

There is a natural adjoint pairing $(C,K)$ so that $C$ is right exact and $K$ is left exact. If $M$ and $N$ are objects in $\cA$, there are canonical exact sequences, $\,KN\hookrightarrow \theta N\onto IN\,$ and $\,DM\hookrightarrow\theta M\onto CM\,$. Each gives rise to a long exact sequence of extensions.

\begin{thm}
\label{GJ}
\cite[5.1.8]{GJ} Suppose that $M$ is $C$-acyclic. There is a natural commutative diagram with exact rows,
\[
\begin{CD}
@>>>\E(M,KN)@>>>\E(M,\theta N)@>>>\E(M,IN)@>{\delta_1}>>\E\tothe{+}\!(M,KN)@>>>\\
@.@V{\gamma}VV@V{\beta}VV@V{\alpha}VV@V{\gamma^+}VV@.\\
@>>>\E(CM,N)@>>>\E(\theta M,N)@>>>\E(DM,N)@>{\delta_2}>>\E\tothe{+}\!(CM,N)@>>>,
\end{CD}
\]
where $\beta$ is an isomorphism. If $DM$ is $C$-acyclic and $IN=N$, then $\alpha$ and $\gamma$ are isomorphisms. 
\end{thm}

\begin{proof}
Let $P\onto M$ be a projective resolution. There is an exact sequence of chain complexes,
\[
\begin{CD}
0@>>>DP@>>>\theta P@>>>CP@>>>0.
\end{CD}
\]
Since $M$ is $C$-acyclic, this is a resolution of the exact sequence,
\begin{equation}
\label{Mseq}
\begin{CD}
0@>>>DM@>>>\theta M@>>>CM@>>>0.
\end{CD}
\end{equation}

Let $X\onto DM$ and $Z\onto CM$ be projective resolutions. Use the horseshoe lemma \cite[2.28]{W} to construct a split exact sequence resolving diagram \eqref{Mseq},
\begin{equation}
\label{Pres}
\begin{CD}
0@>>>X@>>>Y@>>>Z@>>>0.
\end{CD}
\end{equation}
By the comparison theorem \cite[2.3.7]{W}, there are chain maps $a:X\rightarrow DP$ and $c:Z\rightarrow CP$ lifting $\Id_{DM}$ and $\Id_{CM}$ respectively. Using the splitting maps of diagram \eqref{Pres}, construct a chain map $b:Y\rightarrow\theta P$ lifting $\Id_{\theta M}$ and completing a commutative diagram of chain complexes with exact rows,
\[
\begin{CD}
0@>>>X@>>>Y@>>>Z@>>>0\\
@.@VaVV@VbVV@VcVV@.\\
0@>>>DP@>>>\theta P@>>>CP@>>>0.
\end{CD}
\]
Applying $\h(-,N)$ yields a commutative diagram with exact rows,
\begin{equation}
\label{A}
\begin{CD}
0@>>>\h(CP,N)@>>>\h(\theta P,N)@>>>\h(DP,N)\\
@.@VVV@VVV@VVV@.\\
0@>>>\h(Z,N)@>>>\h(Y,N)@>>>\h(X,N)@>>>0
\end{CD}
\end{equation}

Since $P$ is a projective complex, there is also a natural commutative diagram of complexes with exact rows,
\begin{equation}
\label{B}
\begin{CD}
0@>>>\h(P,KN)@>>>\h(P,\theta N)@>>>\h(P,IN)@>>>0\\
@.@VVV@VVV@V\phi_PVV@.\\
0@>>>\h(CP,N)@>>>\h(\theta P,N)@>>>\h(DP,N)\,.
\end{CD}
\end{equation}
The chain map $\phi_P$ is uniquely defined by the diagram because $\h(\pi_P,N)\phi_P=\h(P,\iota_N)$. The first two vertical mappings are isomorphisms.

Combining  diagram \eqref{A} and  diagram \eqref{B}, and applying \cite[1.3.4]{W} yields the Gabber-Joseph diagram. Since $\theta P\onto\theta M$ is a projective resolution, $b$ is a homotopy equivalence so $\beta$ is an isomorphism. (So far, this is the same as the proof given in \cite[5.1.8]{GJ}.)

Let $f:P\rightarrow X$ be a chain map lifting $\pi_M$. Then, by the uniqueness part of the comparison theorem, $af$ is homotopic to $\pi_P$. So,
$\h(f,N)\h(a,N)\phi_P$ is homotopic to $\h(\pi_P,N)\phi_P=\h(P,\iota_N)$. Passing to cohomology, $\E(\pi_M,N)\alpha=\E(M,\iota_N)$.

Now suppose that $DM$ is $C$-acyclic so that $DX\onto DM$ is a resolution. The chain map $D(f):DP\rightarrow DX$ lifts $\Id_{DM}$ so $D(f)a$ is homotopic to $\pi_X$. Hence $\h(a,N)\h(D(f),N)\phi_X$ is homotopic to $\h(\pi_X,N)\phi_X=\h(X,\iota_N)$.
 
By functoriality, $\h(f,N)\h(X,\iota_N)=\h(P,\iota_N)\h(f,IN)$ and, since $\pi$ is a natural transformation, $\h(\pi_P,N)\h(D(f),N)=\h(f,N)\h(\pi_X,N)$.  
Then,
\begin{align*}
\h(\pi_P,N)\h(D(f),N)\phi_X&=\h(f,N)\h(\pi_X,N)\phi_X=\h(f,N)\h(X,\iota_N)\\
&=\h(P,\iota_N)\h(f,IN)=\h(\pi_P,N)\phi_P\h(f,IN)\,.
\end{align*}
Since $\h(\pi_P,N)$ is a monomorphism, $\h(D(f),N)\phi_X=\phi_P\h(f,IN)$ which means that $\h(a,N)\phi_P\h(f,IN)$ is homotopic to $\h(X,\iota_N)$. Passing to cohomology yields $\alpha\E(\pi_M,IN)=\E(DM,\iota_N)$. 

If $IN=N$, $\E(M,\iota_N)=\Id$ and $\E(DM,\iota_N)=\Id$, so that $\alpha$ is an isomorphism. By the long-five lemma, $\gamma$ is also an isomorphism.
\end{proof}

\begin{cor}
\label{GJ'}
If $M$ and $DM$ are $C$-acyclic, then $\E(M,KN)$ and $\E(CM,IN)$ are isomorphic.
\end{cor}

\begin{proof}
By standard properties of adjunction maps, $T(\epsilon_N)$ is an epimorphism.  So $I(\iota_N)$ is an isomorphism as are $\theta(\iota_N)$ and $K(\iota_N)$. In this way,  $I(IN)$, $\theta(IN)$, and $K(IN)$ will be identified with $IN$, $KN$, and $\theta N$ respectively. Applying theorem \ref{GJ} to $IN$, there is a commutative diagram,
\begin{equation*}
\label{GJP}
\begin{CD}
\E(M,IN)@>\delta_1>>\E\tothe+\!(M,KN)\\
@V\alpha'VV@V{\gamma'}VV\\
\E(DM,IN)@>\delta'_2>>\E\tothe+\!(CM,IN),\\
\end{CD}
\end{equation*}
where the vertical mappings are isomorphisms and the primes indicate maps defined with respect to $IN$. 
\end{proof}

\section{The Twisted Sequence}
\label{sec:TS}
 
\begin{thm}
\label{TS}
Suppose that $M$ and $DM$ are $C$-acyclic. There is a commutative diagram with exact rows,
\[
\begin{CD}
@>>>\E\tothe{-}\!(DM,JN)@>\delta>>\E(M,IN)@>\alpha>>\E(DM,N)@>\kappa>>\E(DM,JN)@>>>\\
@.@|@V\delta_1VV@V\delta_2VV@|@.\\
@>>>\E\tothe{-}\!(DM,JN)@>d>>\E\tothe{+}\!(M,KN)%
@>\gamma>>\E\tothe{+}\!(CM,N)@>\chi>>\E(DM,JN)@>>>,
\end{CD}
\]
where $JN=\Coker\epsilon_N$. If $DM=M$, the first row is the long exact sequence associated to the exact sequence, $IN\hookrightarrow N\onto JN$.
\end{thm}

\begin{proof}
Let $\phi'_P$ be the map defined by \eqref{B} with $N=IN$. Then $\h(\pi_P,IN)\phi'_P=\Id$. Using the notation from the previous section,
\begin{align*}
\h(\pi_P,N)\h(DP,\iota_N)\phi'_P&=\h(P,\iota_N)\h(\pi_P,IN)\phi'_P\\
&=\h(P,\iota_N)=\h(\pi_P,N)\phi_P\,.
\end{align*}
Because $\h(\pi_P,N)$ is a monomorphism, $\h(DP,\iota_N)\phi'_P=\phi_P$. Then
\[
\h(a,N)\phi_P=\h(a,N)\h(DP,\iota_N)\phi'_P=\h(X,\iota_N)\h(a,IN)\phi'_P.
\]
Taking cohomology, $\alpha=\E(DM,\iota_N)\,\alpha'$. In a similar fashion, $\beta=\E(\theta M,\iota_N)\,\beta'$ and $\gamma=\E(CM,\iota_N)\,\gamma'$.

\medskip
\noindent{\it Diagram 1\/:}
\begin{equation}
\label{D1}
\begin{CD}
@>>>\E\tothe-\!(DM,JN)@>\delta>>\E(M,IN)@>\alpha>>\E(DM,N)@>\kappa>>\E(DM,JN)@>>>\\
@.@|@V\alpha'VV@|@|@.\\
@>>>\E\tothe-\!(DM,JN)@>>>\E(DM,IN)@>>>\E(DM,N)@>>>\E(DM,JN)@>>>
\end{CD}
\end{equation}

\medskip
\noindent The second row is the long exact sequence associated to $IN\hookrightarrow N\onto JN$. Since $\alpha'$ is an isomorphism, define $\delta$ so that the first square commutes. This produces a commutative diagram with exact rows. If $DM=M$, $\alpha'=\Id$ which proves the second conclusion.

\medskip
\noindent{\it Diagram 2\/:}
\begin{equation}
\label{D2}
\begin{CD}
@>>>\E\tothe-\!(DM,JN)@>>>\E(DM,IN)@>>>\E(DM,N)@>>>\E(DM,JN)@>>>\\
@.@V\delta_3VV@V\delta'_2VV@V\delta_2VV@V\delta_3VV@.\\
@>>>\E(CM,JN)@>>>\E\tothe+\!(CM,IN)@>>>\E\tothe+\!(CM,N)@>>>\E\tothe+\!(CM,JN)@>>>
\end{CD}
\end{equation}

\medskip
\noindent This is a commutative diagram with exact rows where the vertical maps are the natural connecting maps.

\medskip
\noindent{\it Diagram 3\/}:
\begin{equation}
\label{D3}
\begin{CD}
@>>>\E(CM,JN)@>>>\E\tothe+\!(CM,IN)@>>>\E\tothe+\!(CM,N)@>>>\E\tothe+\!(CM,JN)@>>>\\
@.@A\delta_3AA@A\gamma'AA@|@A\delta_3AA@.\\
@>>>\E\tothe-\!(DM,JN)@>d>>\E\tothe+\!(M,KN)@>\gamma>>\E\tothe+\!(CM,N)@>\chi>>\E(DM,JN)@>>>
\end{CD}
\end{equation}

\medskip\noindent
Since $T(\epsilon_N)$ is surjective, $TJN=0$. By the adjoint pairing $(R,T)$, $\E(\theta M,JN)=\E(TM,TJN)=0$ so $\delta_3$ is an isomorphism.
Define $d$ and $\chi$ to make the diagram commutative. Then the second row is also exact.

Assembling the three diagrams proves the first conclusion since $\delta_3^{-1}\,\delta_3=\Id$ and $(\gamma')^{-1}\,\delta'_2\,\alpha'=\delta_1$.\end{proof}

The second row of \ref{TS} will be referred to as a twisted sequence. 

\section{Applications in category $\cO$: older results}
\label{sec:O}

Let $\fg$ be a finite-dimensional semisimple Lie algebra over $\C$. Category $\cO$ is the category of $\fg$-modules introduced in \cite{BGG}. For background information on category $\cO$, we will rely on \cite{H1} where the original sources and the later developments can be found.

Let $S$ be the set of simple root reflections in the Weyl group $W$. The stabilizer of a weight $\la$ under the dot action is $W_\la^\circ$. Let $w_0$ denote the longest element and let $1$ denote the identity. The Bruhat order on $W$ is denoted by $<\,$. Let $\xi$ be its characteristic function defined by
\begin{equation*}
\xi(x,y)=
\begin{cases}
1&\text{if $x\le y$ and}\cr
0&\text{otherwise.}
\end{cases}
\end{equation*}
Let $\ell$ denote the length function on $W$. If $x,y\in W$, $\ell(x,y)=\ell(y)-\ell(x)$.

The $R$-polynomials are defined in \cite[\S 7]{H2}. Let $r_p(x,y)$ denote the coefficient of $q^p$ in $(-1)^{n-p}R_{\,x,\,y}\,$ where $n=\ell(x,y)$.  A recursion for $r_p(x,y)$ begins with  $r_p(w_0,y)=0$ if $p\ne0$ and $r_0(w_0,y)=\xi(w_0,y)$. If $x<w_0$, choose an $s\in S$ so that $xs>x$. Then, for all $p$,
\begin{equation}
\label{recurse}
r_p(x,y)=
\begin{cases}
r_p(xs,ys)&\text{ if $ys>y$,}\cr
r_p(xs,y)+r_{p-1}(xs,y)-r_{p-1}(xs,ys)&\text{ if $ys<y$.}
\end{cases}
\end{equation}

The following properties of the $r_p$ can be proved by induction or translated from properties of the $R$-polynomials in \cite[\S 7]{H2}. If $r_p(x,y)\ne0$, then $x\le y$ and $0\le p\le \ell(x,y)$. Also $r_0=\xi$ and, if $n=\ell(x,y)$, $r_p(x,y)=r_{n-p}(x,y)$.

Specializing \eqref{recurse} to $p=1$, $r_1(w_0,y)=0$ and, if $xs>x$,
\begin{equation}
\label{recurse2}
r_1(x,y)=
\begin{cases}
r_1(xs,ys)&\text{ if $ys>y$,}\cr
r_1(xs,y)+1&\text{ if $ys<y$ and $xs\,{\not<}\,ys$,}\cr
r_1(xs,y)&\text{ if $ys<y$ and $xs<ys$.}
\end{cases}
\end{equation}

Choose anti-dominant integral weights $\la$ and $\mu$ so that $W_\la^\circ{=}\{e\}$ and $W_\mu^\circ=\{e,s\}$ where $s\in S$. If $x\in W$, let $M_x$ denote the Verma module with highest weight $x\cdot\la$. The block of $\cO$ with projective generator $M_{w_0}$ is $\cO_\la$ \cite[4.9]{H1}.
Here, $T$ is the translation functor $T_\la^\mu$ where $R$ is its left and right adjoint $T_\mu^\la$ \cite[7.1-2]{H1}. A module $M\,{\in}\,\cO_\la$ is $C$-acyclic if, and only if, $DM=M$ \cite[2.9]{C} and this condition is true for each $M_x$ \cite[2.8(i)]{C}.

For $x,y\in W$, write $\E\tothe p(x,y)$ for $\E\tothe p(M_x,M_y)$ and $e_p(x,y)$ for its dimension. Also, for $x$ and $z\le y$ in $W$, write $\E\tothe p(x,y/z)$ for $\E^{\,p}(M_x,M_y/M_z)$ and let $e_p(x,y/z)$ be the dimension.

Since $M_{w_0}$ is projective, $e_p(w_0,y)=0$ if $p\ne0$. By the properties of homomorphisms between Verma modules, $e_0=\xi$ \cite[5.2]{H1}, so $e_0=r_0$.  The vanishing properties also match. If $e_p(x,y)\ne0$ then $x\le y$ and $0\le p\le\ell(x,y)$ \cite[6.11]{H1}.

The twisted sequence can be used to re-prove some of the results of \cite[5.2]{GJ}.

\begin{prop}\cite[5.2.1]{GJ}
\label{GJ1}Suppose that $xs>x$ and $ys<y$. For all $p$, \[e_p(xs,y)=e_p(x,ys).\]
\end{prop}

\begin{proof}Let $M=M_{xs}$ and $N=M_{ys}$.
Then $CM=M_x$, $IN=N$, and $KN=M_y$ \cite[3.5]{C}. By \ref{GJ'}, $\E(xs,y)$ is isomorphic to $\E(x,ys)$.
\end{proof}

Suppose that $xs>x$ and $ys<y$. Apply \ref{TS} with $M=M_{xs}$ and $N=M_y$. Then 
$IN=M_{ys}$, $CM=M_x$, and $KN=M_y$. There is a commutative diagram with exact rows,
\begin{equation}
\label{TS1}
\begin{CD}
@>>>\E\tothe{-}\!(xs,y/ys)@>\delta>>\E(xs,ys)@>\alpha>>\E(xs,y)@>\kappa>>\E(xs,y/ys)@>>>\\
@.@|@V\delta_1VV@V\delta_2VV@|@.\\
@>>>\E\tothe{-}\!(xs,y/ys)@>d>>\E\tothe{+}\!(xs,y)%
@>\gamma>>\E\tothe{+}\!(x,y)@>\chi>>\E(xs,y/ys)@>>>.
\end{CD}
\end{equation}
The following result is the twisted equivalent of \cite[5.2.3]{GJ}.

\begin{prop}
\label{GJ2}Suppose that $xs>x$ and $ys<y$. For all $p$,
\[e_p(x,y)-e_p(xs,y)\ge e_{p-1}(xs,y)-e_{p-1}(xs,ys)\]
and this is an equality if, and only if, $\Ker d^{\,p-1}=\Ker\delta^{\,p-1}$ and 
$\Ker d^{\,p-2}=\Ker\delta^{\,p-2}$.
\end{prop}

\begin{proof}Since $d=\delta_1\,\delta$, $\Ker\delta\subseteq\Ker d$. Identify $\E$ with $\E\tothe{p-1}$ and $d$ with $d^{\,p-2}$ in diagram \eqref{TS1}. Because the second row is exact, there is a short exact sequence
\begin{equation*}
\begin{CD}
0@>>>\Image d^{\,p-2}@>>>\E\tothe{p}(xs,y)@>>>\E\tothe{p}(x,y)@>>>\Ker d^{\,p-1}@>>>0.
\end{CD}
\end{equation*}
Then
\begin{align*}
e_p(x,y)-e_p(xs,y)&=\dim\Ker d^{p-1}-(e_{p-2}(xs,y/ys)-\dim\Ker d^{p-2})\cr
&\ge\dim\Ker \delta^{p-1}-(e_{p-2}(xs,y/ys)-\dim\Ker \delta^{p-2})\cr
&=e_{p-1}(xs,y)-e_{p-1}(xs,ys),
\end{align*}
where the last equality uses the exactness of the first row of \eqref{TS1}.
\end{proof}

\begin{cor}\cite[3.9]{C}
\label{GJ3}
Suppose that $xs>x$, $ys<y$, and $xs\not<ys$. For all $p$, 
\[e_p(x,y)=e_p(xs,y)+e_{p-1}(xs,y).\]
\end{cor}

\begin{proof}
Because $\E(xs,ys)=0$, $\delta=0$ and $d=0$. The conditions for equality in \ref{GJ2} are satisfied.
\end{proof}

These results led naturally to the conjecture that $e_p=r_p$ for all $p$ \cite[3.1]{C}. It was soon discovered that there are examples where $r_p(x,y)$ is negative \cite{B}, so equality in \ref{GJ2} can not hold in general. One easy consequence of \cite[5.2.3]{GJ} is that $r_1$ is, at least, a lower bound for $e_1$. (Later, it will be shown that $e_1\ne r_1$.)

\begin{prop}
\label{e1r1}
$e_1\ge \,r_1$
\end{prop}

\begin{proof}Assume there is a counterexample, $e_1(x,y)<r_1(x,y)$, with $x$ maximal in the Bruhat ordering.
If $x=w_0$, $e_1(w_0,y)=0=r_1(w_0,y)$ so  $x<w_0$. Choose an $s\in S$ with $xs>x$. There are two cases to consider.

If  $ys>y$, then  $e_1(x,y)=e_1(xs,ys)$ by \ref{GJ1}. Since $x$ is maximal, $e_1(xs,ys)\ge r_1(xs,ys)=r_1(x,y)$ by equation \eqref{recurse2}.

If $ys<y$, then \ref{GJ2} implies that $e_1(x,y)\ge e_1(xs,y)+e_0(xs,y)-e_0(xs,ys)$. Since $e_0=r_0$ and $x$ is maximal, 
$e_1(x,y)\ge r_1(xs,y)+r_0(xs,y)-r_0(xs,ys)=r_1(x,y)$ by \eqref{recurse}.

In either case, $e_1(x,y)\ge r_1(x,y)$, which contradicts the choice of $x$.
\end{proof}

The twisted sequence in diagram \eqref{TS1} has the same terms as the two-line spectral sequence of \cite[3.4]{C}. It is an indirect resolution of the conjecture that the coboundary of the spectral sequence should factor as $d=\delta_1\,\delta$ \cite[p. 37]{C}. It can also be substituted for the spectral sequence in many of the proofs. As an example, one result that is needed below will be re-proved here.

\begin{prop}
\label{Cn}
\cite[3.8]{C}
If $x\le y$ and $n=\ell(x,y)$, then $e_n(x,y)=1$.
\end{prop}

\begin{proof}Suppose that $x\le y$ and assume that there is a counterexample with $x$ maximal. If $x=w_0$, then $y=w_0$, $n=0$ and $e_0(w_0,w_0)=1$ so  $x<w_0$. Choose an $s\in S$ so that $xs>x$.  There are two cases to consider.

If $ys>y$, and $e_n(x,y)=e_n(xs,ys)$ by \ref{GJ1}. Because $x$ is maximal and $xs\le ys$, $e_n(xs,ys)=1$.

If $ys<y$,  then consider diagram \eqref{TS1} with $\E=\E\tothe{n-1}$ and apply the vanishing properties.
\begin{equation*}
\begin{CD}
@>>>\E\tothe{-}\!(xs,y/ys)@>\delta>>0@>>>\E(xs,y)@>>>\E(xs,y/ys)@>>>0\\
@.@|@VVV@V\delta_2VV@|@.\\
@>>>\E\tothe{-}\!(xs,y/ys)@>d>>0%
@>>>\E\tothe{+}\!(x,y)@>>>\E(xs,y/ys)@>>>0.
\end{CD}
\end{equation*}
Then $\delta_2$ is an isomorphism, so $e_n(x,y)=e_{n-1}(xs,y)$.  But $e_{n-1}(xs,y)=1$ since $xs\le y$ and $x$ is maximal.

In either case,  $e_n(x,y)=1$, which contradicts the choice of $x$.
\end{proof}

In the remainder of this section, the recursive calculation of $e_{n-1}(x,y)$ where $n=\ell(x,y)$ will be considered. Suppose that $x<xs<ys<y$ for some $s\in S$. Applying diagram \eqref{TS1} with $\E=\E\tothe{n-2}$ yields
\begin{equation}
\label{Dn-2}
\begin{CD}
@>>>\E\tothe{-}\!(xs,y/ys)@>\delta>>\E(xs,ys)@>>>\E(xs,y)@>>>\E(xs,y/ys)@>>>0\\
@.@|@V\delta_1VV@VVV@|@.\\
@>>>\E\tothe{-}\!(xs,y/ys)@>d>>\E\tothe{+}\!(xs,y)%
@>>>\E\tothe{+}\!(x,y)@>>>\E(xs,y/ys)@>>>0.
\end{CD}
\end{equation}
By \ref{Cn}, $e_{n-2}(xs,ys)=e_{n-1}(xs,y)=1$ so that $\delta_1$ is an isomorphism or zero.
But $\delta_1$ is part of the exact sequence
\begin{equation*}
\begin{CD}
\E\tothe{n-2}\!(xs,ys)@>\delta_1>>\E\tothe{n-1}\!(xs,y)@>>>\E\tothe{n-1}\!(M_{xs},\theta M_y)@>>>0,
\end{CD}
\end{equation*}
showing that $\delta_1$ is an isomorphism, if and only if, $\E\tothe{n-1}(M_{xs},\theta M_y)$ is zero. By the adjoint pairing $(T,R)$, $\E\tothe{n-1}(M_{xs},\theta M_y)$ is isomorphic to $\E\tothe{n-1}(TM_{xs},TM_y)$. The vanishing behavior of this singular extension group determines whether $d$ is zero or surjective. This suggests a conjecture on singular vanishing.

\begin{conj}
\label{C3}
If $x<xs<ys<y$, then $\E\tothe{n-1}\!(TM_{xs},TM_y)=0$, where $n=\ell(x,y)$. 
\end{conj}

\begin{prop}
\label{Cn-1}Suppose that $x<y$ and let $n=\ell(x,y)$. Conjecture \ref{C3} implies that  \[e_{n-1}(x,y)=r_1(x,y).\]
\end{prop}

\begin{proof}Assume there is a counterexample with $x$ maximal. Because $y\le w_0$, $x<w_0$ and there is an $s\in S$ with $xs>x$. There are three cases to consider.

If $ys>y$, \ref{GJ1} implies that $e_{n-1}(x,y)=e_{n-1}(xs,ys)$. Since $x$ is maximal and $xs<ys$ , $e_{n-1}(xs,ys)=r_1(xs,ys)=r_1(x,y)$ by equation \eqref{recurse2}.

If $ys<y$ and $xs\not<ys$, $e_{n-1}(x,y)=e_{n-1}(xs,y)+e_{n-2}(xs,y)$ by \ref{GJ3}. Since $xs\le y$, $e_{n-1}(xs,y)=1$ by \ref{Cn}.  If $xs=y$, then $n=1$ and $e_0(x,y)=r_0(x,y)=r_1(x,y)$  so $xs<y$ by the choice of $x$. Because $x$ is maximal, $e_{n-2}(xs,y)=r_1(xs,y)$. Then  $e_{n-1}(x,y)=1+r_1(xs,y)=r_1(x,y)$ by equation \eqref{recurse2}.

If $x<xs<ys<y$
and assuming that conjecture \ref{C3} is true, $\delta_1$ in diagram \eqref{Dn-2} is an isomorphism. Then $e_{n-1}(x,y)=e_{n-2}(xs,y)$. Because $x$ is maximal, $e_{n-2}(xs,y)=r_1(xs,y)=r_1(x,y)$ by equation \eqref{recurse2}.

In each case, $e_{n-1}(x,y)=r_1(x,y)$, which contradicts the choice of $x$.
\end{proof}

\section{Applications in category $\cO$: younger results}
\label{sec:ext1}

Most of the results of the last section have been known for a long time. The newer results involve $r_1$. The first new result in this direction was published by Mazorchuk in 2007.

\begin{prop}
\cite[Lemma 33]{M}
\label{thm:maz}
$e_1(1,w_0)=\left|S\right|\,.$
\end{prop}

\begin{cor}
\label{cor:maz}For all $x,y\in W$,
\begin{enumerate}
\item$e_1(x,w_0)=r_1(x,w_0)\,$ and
\item$e_1(1,y)=r_1(1,y)\,$.
\end{enumerate}
\end{cor}

The first item of \ref{cor:maz} is equivalent to the original statement of \cite[Theorem 32]{M} (adjusting for anti-dominance and ignoring the grading). It is expressed here in terms of $r_1$. The proof of the corollary uses the following lemma.

\begin{lem}
Suppose that $xs>x$ and $ys<y$ for some $s\in S$. If $e_1(x,y)=r_1(x,y)$, then $e_1(xs,y)=r_1(xs,y)$
\end{lem}

\begin{proof}Suppose that  $e_1(xs,y)\ne r_1(xs,y)$. By \ref{e1r1}, $e_1(xs,y)>r_1(xs,y)$. Using \ref{GJ2} and \ref{recurse},
\begin{align*}
e_1(x,y)&\ge e_1(xs,y)+e_0(xs,y)-e_0(xs,ys)\\
&>r_1(xs,y)+r_0(xs,y)-r_0(xs,ys)=r_1(x,y),
\end{align*}
so $e_1(x,y)\ne r_1(x,y)$
\end{proof}

\begin{proof}[Proof of the corollary]
To show that $e_1(1,w_0)=r_1(1,w_0)$, apply \cite[7.10(20)]{H2} with $x=1$ and $w=w_0$ to get
\begin{equation*}
\sum_{1\le y\le w_0}R_{1,y}=q^n,
\end{equation*}
where $n=\ell(1,w_0)$. The coefficient of $q^{n-1}$ on the left-hand side is
\[
(-1)^1r_{n-1}(1,w_0)+\left|S\right|,
\] 
so
$r_1(1,w_0)=r_{n-1}(1,w_0)=\left|S\right|$.

To prove item (i), assume that there is a counterexample with $x$ minimal. Then $x>1$ and there is an $s\in S$ with $xs<x$.  By minimality of $x$,  $e_1(xs,w_0)=r_1(xs,w_0)$. The lemma implies that
$e_1(x,w_0)=r_1(x,w_0)$, contradicting the choice of $x$.

The proof of item (ii) is similar.
\end{proof}

The next development was Noriyuki Abe's preprint that originally appeared on the ArXiv in 2010. Let $v(x,y)=e_1(x,y)-e_0(x,w_0/y)$ if $x\le y$ and let $v(x,y)=0$ if $x\not\le y$. If $x\le y$, then $v(x,y)=\dim V(w_0x,w_0y)$ in Abe's notation. Then \cite[theorem 4.4]{A1} becomes $v=r_1$. As stated, the theorem is not true. There are $16$ pairs $(x,y)$ in type $\text{B}_3$ with $r_1(x,y)=4$ but, by definition, $v\le3$ \cite[Theorem 1.1(1)]{A1}. Abe's recursion for $V$ \cite[Theorem 4.3]{A1} does imply that $v\le r_1$ (by comparison with \ref{recurse2}). Then, combined with \ref{e1r1}, $v\le r_1\le e_1$ or
\begin{equation*}
r_1(x,y)\le e_1(x,y)\le r_1(x,y)+e_0(x,w_0/y).
\end{equation*}
Note that $e_0(1,w_0/y)=0$ and $e_0(x,w_0/w_0)=0$, so Abe's inequality does generalize \ref{cor:maz}. Although $v\ne r_1$, Abe has communicated an example in type $\text{B}_3$ showing that $e_1\ne r_1$ \cite{A2}.
 
In the remainder, the twisted sequence approach will be used to prove properties of $v$ that correspond with Abe's results from \cite{A1}.

\begin{prop}
\label{AGJ1}
If $xs>x$ and $ys<y\,$, then $e_0(xs,w_0/y)=e_0(x,w_0/ys)$.
\end{prop}

\begin{proof}
Let $M=M_{xs}$ and $N=M_{w_0}/M_{ys}$. There is a commutative diagram with exact rows,
\begin{equation}
\label{snake}
\begin{CD}
0@>>>\theta M_{ys}@>>>\theta M_{w_0}@>>>\theta N@>>>0\\
@.@VVV@VVV@VVV@.\\
0@>>>M_{ys}@>>>M_{w_0}@>>>N@>>>0.\\
\end{CD}
\end{equation}
By the snake lemma, $KN=M_{w_0}/M_y$ . By the adjoint pairing $(C,K)$, $\h(M_{xs},KN)$ and $\h(M_x,N)$ are isomorphic.
\end{proof}

By \ref{GJ1}, if $xs>x$ and $ys<y\,$, $e_1(xs,y)=e_1(x,ys)$ which proves the following property of $v$, which corresponds to \cite[4.3(1)]{A1}.

\begin{cor}
If $xs>x$ and $ys<y\,$, then $v(xs,y)=v(x,ys)$.
\end{cor}

Next, there is another ladder diagram that links extensions of fractional Verma modules to the twisted sequence.

\begin{prop}
\label{FTS}
Suppose that $xs>x$ and $ys<y\,$. There is a commutative diagram with exact rows,
\begin{equation*}
\label{LD}
\begin{CD}
@>>>\E\tothe{-}\!(xs,y/ys)@>\delta>>\E(x,w_0/ys)@>\alpha>>\E(x,w_0/y)@>\kappa>>\E(xs,y/ys)@>>>\\
@.@|@V\delta_1VV@V\delta_2VV@|\\
@>>>\E\tothe{-}\!(xs,y/ys)@>d>>\E\tothe{+}\!(xs,y)@>\gamma>>\E\tothe{+}\!(x,y)@>\chi>>\E(xs,y/ys)@>>>,\\
\end{CD}
\end{equation*}
where the second row is the same as the second row of diagram \eqref{TS1}.
\end{prop}

\begin{proof}
The proof is similar in structure to the proof of \ref{TS}. Fix a commuting triangle of Verma module injections,
\begin{equation}
\label{tri}
\begin{CD}
M_{ys}@>>>M_y\\
@|@VVV\\
M_{ys}@>>>M_{w_0}.\\
\end{CD}
\end{equation}

\medskip
\noindent{\it Diagram 1\/:}
\begin{equation*}
\label{DA1}
\begin{CD}
@>>>\E\tothe-\!(xs,y/ys)@>\delta>>\E(x,w_0/ys)@>\alpha>>\E(x,w_0/y)@>\kappa>>\E(xs,y/ys)@>>>\\
@.@V\delta_3VV@|@|@V\delta_3VV@.\\
@>>>\E(x,y/ys)@>>>\E(x,w_0/ys)@>>>\E(x,w_0/y)@>>>\E\tothe{+}\!(x,y/ys)@>>>
\end{CD}
\end{equation*}

\medskip
\noindent The map $\delta_3$ is the same as the isomorphism $\delta_3$ from diagram \eqref{D2} with $M=M_{xs}$ and $N=M_y$. The second row is the long exact sequence associated to the exact sequence,
\[M_y/M_{ys}\hookrightarrow M_{w_0}/M_{ys}\onto M_{w_0}/M_y.\]
Define $\delta$ and $\kappa$ so that the diagram commutes. This produces a commutative diagram with exact rows.

\medskip
\noindent{\it Diagram 2\/:}
\begin{equation*}
\label{DA2}
\begin{CD}
@>>>\E(x,y/ys)@>>>\E(x,w_0/ys)@>>>\E(x,w_0/y)@>\delta_7>>\E\tothe{+}\!(x,y/ys)@>>>\\
@.@|@V\delta_5VV@V\delta_6VV@|@.\\
@>>>\E(x,y/ys)@>\delta_4>>\E\tothe+\!(x,ys)@>>>\E\tothe+\!(x,y)@>>>\E\tothe+\!(x,y/ys)@>>>
\end{CD}
\end{equation*}

\medskip
\noindent This is a commutative diagram with exact rows where $\delta_k$, $4\le k\le7$ are natural connecting maps (all derived from rotations of diagram \eqref{tri}). For example, the middle square commutes because of the short ladder,
\begin{equation*}
\begin{CD}
0@>>>M_{ys}@>>>M_{w_0}@>>>M_{w_0}/M_{ys}@>>>0\\
@.@VVV@|@VVV@.\\
0@>>>M_y@>>>M_{w_0}@>>>M_{w_0}/M_y@>>>0.\\
\end{CD}
\end{equation*}

\medskip
\noindent{\it Diagram 3\/}:
\begin{equation*}
\label{DA3}
\begin{CD}
@>>>\E(x,y/ys)@>>>\E\tothe+\!(x,ys)@>>>\E\tothe+\!(x,y)@>>>\E\tothe+\!(x,y/ys)@>>>\\
@.@A\delta_3AA@A\gamma'AA@|@A\delta_3AA@.\\
@>>>\E\tothe-\!(xs,y/ys)@>d>>\E\tothe+\!(xs,y)@>\gamma>>\E\tothe+\!(x,y)@>\chi>>\E(xs,y/ys)@>>>
\end{CD}
\end{equation*}

\medskip\noindent
This is a commutative diagram with exact rows because it is diagram \eqref{D3} with $M=M_{xs}$ and $N=M_y$. Since $\gamma'$ is an isomorphism, assembling the diagrams completes the proof.
\end{proof}

Applying the same argument as in the proof of \ref{GJ2} yields the following inequality.

\begin{prop}
\label{AGJ2}Suppose that $xs>x$ and $ys<y$. For all $p$,
\[e_p(x,y)-e_p(xs,y)\ge e_{p-1}(x,w_0/y)-e_{p-1}(x,w_0/ys).\]
This is an equality if, and only if, $\Ker d^{\,p-1}=\Ker\delta^{\,p-1}$ and 
$\Ker d^{\,p-2}=\Ker\delta^{\,p-2}$.
\end{prop}

\begin{cor}
\label{AGJ3}
If $xs>x$ and $ys<y$, then 
\[
e_1(x,y)-e_1(xs,y)\ge e_0(x,w_0/y)-e_0(xs,w_0/y)
\]
and this is an equality if, and only if, $\Ker d^{\,0}=\Ker \delta^{\,0}$
\end{cor}

\begin{proof}
Taking $p=1$ in \ref{AGJ2},
\[
e_1(x,y)-e_1(xs,y)\ge e_0(x,w_0/y)-e_0(x,w_0/ys).
\]
By \ref{AGJ1}, $e_0(x,w_0/ys)=e_0(xs,w_0/y)$.
\end{proof}

The conclusion is equivalent to $v(x,y)\ge v(xs,y)$. When $xs<ys$, Abe proves $v(x,y)=v(xs,y)$ by showing that the images of $\E^1(xs,y)$ and $\E^1(x,y)$ in $\E^1(x,w_0)$ are the same \cite[4.3(2)]{A1}.

The preceding proposition is sufficient, by itself, to explain Abe's counter-example for $e_1=r_1$. In type $\text{B}_3$, let $s_1$, $s_2$, and $s_3$ be the simple root reflections, where $s_1s_2$ has order $3$ and $s_2s_3$ has order $4$. Take $x=s_1s_3$, $y=w_0s_3=s_2s_3s_1s_2s_3s_2s_1s_2$, and $s=s_2$. Using the work of H. Matumoto \cite{Mat} on scalar, generalized Verma module homomorphisms, Abe shows that there is a nonzero homomorphism between $M_x$ and $M_{w_0}/M_y$ so $e_0(x,w_0/y)\ne0$  \cite{A2}. Kazhdan-Lusztig multiplicities imply that $e_0(x,w_0/y)-e_0(xs,w_0/y)=1$. By \ref{AGJ3}, $e_1(x,y)>e_1(xs,y)$, which means $e_1(x,y)\ne r_1(x,y)$.
 
\begin{prop}
\label{AGJ4}
Suppose that $x<xs\le y$ and $ys<y$. If $xs\not<ys$, then $v(x,y)\le v(xs,y)+1$ and this is an equality if, and only if, $\Ker\delta^{\,0}=0$.
\end{prop}

\begin{proof}
In \ref{FTS}, $d=0$ by \ref{GJ3}. Also $e_0(xs,y/ys)=1$ implies that $e_0(x,w_0/y)-e_0(x,w_0/ys)\le1$.
\end{proof}

The condition for equality in \ref{AGJ4} must somehow be equivalent to the condition $v_s\not\in sV(w_0xs,w_0y)$ from \cite[4.3(2)]{A1}. Finally, another twisted sequence can be used to prove a result that is also consistent with \cite[4.3(2)]{A1}.

Suppose that $xs>x$ and $ys<y$. Let $M=M_{xs}$ and $N=M_{w_0}/M_{ys}$. There is a twisted sequence associated to $N$. From diagram \eqref{snake},
$IN=M_{w_0s}/M_{ys}$ so, by \ref{TS}, there is  a commutative diagram with exact rows,
\begin{equation}
\label{SLD}
\begin{CD}
@>>>\E\tothe{-}\!(xs,\frac{w_0}{w_0s})@>\delta>>\E(xs,\frac{w_0s}{ys})@>\alpha>>\E(xs,\frac{w_0}{ys})@>\kappa>>\E(xs,\frac{w_0}{w_0s})@>>>\\
@.@|@V\delta_1VV@V\delta_2VV@|\\
@>>>\E\tothe{-}\!(xs,\frac{w_0}{w_0s})@>d>>\E\tothe{+}\!(xs,\frac{w_0}{y})@>\gamma>>\E\tothe{+}\!(x,\frac{w_0}{ys})@>\chi>>\E(xs,\frac{w_0}{w_0s})@>>>.\\
\end{CD}
\end{equation}

\medskip
\begin{prop}
\label{last}
Suppose that $x<xs\le y$ and $ys<y$. If $xs\not<w_0s$, then $v(x,y)=v(xs,y)+1$.
\end{prop}

\begin{proof}
Because $ys<w_0s$, $xs\not<w_0s$ implies $xs\not<ys$ and hence $e_0(xs,w_0s/ys)=0$. If $\E$ is identified with $\E^0$ in diagram \eqref{SLD}, $\kappa$ is an injective map, which implies that $\delta_2$ is injective. Working through the definitions, there is a commutative diagram,
\begin{equation*}
\begin{CD}
0@>>>\h(xs,y/ys)@>>>\h(xs,w_0/ys)\\
@.@V\delta VV@V\delta_2VV\\
@.\E^1(x,w_0/ys)@=\E^1(x,w_0/ys),\\
\end{CD}
\end{equation*}
where $\delta$ is the homomorphism defined in the proof of \ref{FTS}. Since $\delta_2$ is injective, $\Ker\delta=0$ and $v(x,y)=v(xs,y)+1$ by \ref{AGJ4}.
\end{proof}

In a similar vein, one can prove that $v(x,y)=v(xs,y)$ if $x<xs<ys<y$ and $e_0(xs,w_0s/ys)=0$. In that case, $e_1(x,y)=e_1(xs,y)$ as well.

If the goal is a general recursive formula for $e_1$, then the goal is well over the horizon. The classic conjecture, $e_1=r_1$, is false. Abe's recursion for $v$ is very effective (and $v$ is bounded above by the rank of $\fg$), but the resulting determination of $e_1$ depends on the very difficult problem of generalized Verma module homomorphisms. If $x\le y$ and $e_0(x,w_0/y)$ is known, then $e_1(x,y)=v(x,y)+e_0(x,w_0/y)$.

\end{document}